\documentclass[12pt]{amsart}

 \usepackage[dvips]{graphicx}
\usepackage{amsmath, amssymb}
\usepackage{pifont}

\numberwithin{equation}{section}
\newtheorem{theorem}{Theorem}[section]  
\newtheorem{theorem?}{``Theorem''}[section]  
\newtheorem{corollary}[theorem]{Corollary}

\newtheorem{lemma}[theorem]{Lemma}

\theoremstyle{definition}

\newtheorem{question}{Question}

\theoremstyle{remark}
\newtheorem{remark}[theorem]{Remark}  

\newcommand{\R}{{\mathbb R}}
\newcommand{\C}{{\mathbb C}}

\newcommand{\N}{{\mathbb N}}
\newcommand{\Z}{{\mathbb Z}}

\renewcommand{\a}{\alpha}

\begin{document}
\title[local zeta functions]
{
Non-polar singularities of 
local zeta functions \\
in some smooth case} 
\author{Joe Kamimoto and Toshihiro Nose} 
\dedicatory{Dedicated to Professor Takeo Ohsawa on the occasion of 
his retirement.}
\address{Faculty of Mathematics, Kyushu University, 
Motooka 744, Nishi-ku, Fukuoka, 819-0395, Japan} 
\email{
joe@math.kyushu-u.ac.jp}
\address{Faculty of Engineering, Fukuoka Institute of Technology, 
Wajiro-higashi 3-30-1, Higashi-ku, Fukuoka, 811-0295, Japan}
\email{ 
nose@fit.ac.jp}
\maketitle


\begin{abstract}
It is known that local zeta functions
associated with real analytic functions 
can be analytically continued as meromorphic functions 
to the whole complex plane.
In this paper, 
the case of specific  (non-real analytic) smooth functions
is precisely investigated. 
Indeed, 
asymptotic limits of the respective local zeta functions 
at some singularities in one direction are 
explicitly computed. 
Surprisingly, 
it follows from these behaviors that 
these local zeta functions 
have singularities different from poles.  
\end{abstract}




\section{Introduction}

Let us consider the following integrals of the form:
\begin{equation}\label{eqn:1.1}
Z_f(\varphi)(s)=\int_{\mathbb{R}^2}|f(x,y)|^s \varphi(x,y)dxdy
\quad\quad \mbox{for $s\in \C$,}
\end{equation}
where
$f,\varphi$ are real-valued ($C^{\infty}$) smooth functions 
defined on an open neighborhood $U$ of the origin 
in $\mathbb{R}^2$
and the support of $\varphi$ is contained in $U$.
Since the integrals 
$Z_f(\varphi)(s)$ 
converge locally uniformly 
on the region: ${\rm Re}(s)>0$,
they become holomorphic functions there.
It has been known in many cases that
they can be holomorphically continued to wider regions. 
After this process, these integrals become holomorphic functions 
on domains containing the region: ${\rm Re}(s)>0$, which
are sometimes called {\it local zeta functions}.  
In this paper, we are interested in the case when 
$f$ satisfies the condition:
\begin{equation}\label{eqn:1.2}
f(0,0)=0, \quad \nabla f(0,0)=(0,0)
\end{equation}
and hereafter we always assume this condition. 
(Our issues in this paper are easy 
unless (\ref{eqn:1.2}) is satisfied.)

In order to see the region where $Z_f(\varphi)$ can be holomorphically 
continued, the following index plays important roles: 
\begin{equation}\label{eqn:1.3}
c_0(f):=\sup
\left\{\mu>0:
\begin{array}{l} 
\mbox{there exists an open neighborhood $V$ of} \\
\mbox{the origin in $U$ such that $|f|^{-\mu}\in L^1(V)$}
\end{array}
\right\}.
\end{equation}
This  $c_0(f)$ is called the {\it critical integrability index} of $f$
(it is also called {\it log canonical threshold} or 
{\it singularity exponent}).
The determination of the value of $c_0(f)$ 
is an important issue in the singularity theory and 
there have been many interesting works from many points of view 
(this problem will be discussed soon later).
In order to see a clear relationship between the index 
$c_0(f)$ and the region where $Z_f(\varphi)(s)$ becomes holomorphic, 
we assume in the Introduction 
that $\varphi$ satisfies the condition:
\begin{equation}\label{eqn:1.4}
\varphi(0,0)>0, \quad \varphi(x,y)\geq 0 \quad \mbox{ on $U$}.
\end{equation}
Indeed, under this condition (\ref{eqn:1.4}), 
the relationship between the convergence of 
the integrals (\ref{eqn:1.1}) and their holomorphy
implies the equality:
\begin{equation}\label{eqn:1.5}
c_0(f)=\sup
\left\{\rho>0:
\begin{array}{l} 
\mbox{The domain where $Z_f(\varphi)$ can be holomorphically} \\
\mbox{continued contains the region: ${\rm Re}(s)>-\rho$}
\end{array}
\right\}
\end{equation}
Without the condition (\ref{eqn:1.4}), 
the right handside of (\ref{eqn:1.5}) 
may be greater than $c_0(f)$
(see \cite{KaN16tams}, etc.).

The above equality (\ref{eqn:1.5}) means that 
$Z_f(\varphi)$ is holomorphic 
on the region: ${\rm Re}(s)>-c_0(f)$ and, moreover, that  
$Z_f(\varphi)$ has some singularities on the vertical line: 
${\rm Re}(s)=-c_0(f)$.
More exactly, 
the following lemma implies that
$Z_f(\varphi)$ must have a singularity at $s=-c_0(f)$.
\begin{lemma}\label{lemma:1.1}
$Z_f(\varphi)$ cannot be holomorphically continued
to any open neighborhood of $s=-c_0(f)$.
\end{lemma}
The proof of the above lemma will be given in Section~6. 
The purpose of this paper is to consider the following question:
\begin{question}
What kind of singularity does $Z_f(\varphi)$ have at $s=-c_0(f)$?
\end{question}

Under certain assumptions of $f$,
the integrals $Z_f(\varphi)(s)$
have meromorphic continuation to the whole complex plane and, 
in particular, they have a pole at $s=-c_0(f)$.
A brief history of the studies done on
this phenomena
is as follows.
For a meanwile, 
the general dimensional cases are treated.   
In 1954 in an invited talk at ICM Amsterdam, 
I. M. Gel'fand conjectured that
if $f$ is a polynomial and the support of $\varphi$ is
sufficiently small,
then 
the integrals $Z_f(\varphi)(s)$
can be analytically continued as meromorphic functions
to the whole complex plane.
The case when $f$ are monomials is investigated in \cite{GeS64}.
Gel'fand's conjecture was affirmatively solved as
a stronger form by
Bernstein, S. I. Gel'fand \cite{BeG69}
and  Atiyah \cite{Ati70} independently.
They showed that
if $f$ is real analytic
and
the support of $\varphi$ is sufficiently small,
then
$Z_f(\varphi)$ can be analytically
continued as a meromorphic function
to the whole complex plane and
their poles belong to the union of finite number of arithmetic progressions
which consist of negative rational numbers.
Their proofs use Hironaka's resolution of singularities 
\cite{Hir64}.
The authors \cite{KaN16jmsut}  
generalize these results 
when $f$ belongs to some class of smooth functions.

More precisely, let us consider an issue about 
the determination of the value of $c_0(f)$.
We are also interested in more detailed local behavior of 
$Z_f(\varphi)$ near $s=-c_0(f)$. 
In the real analytic case, 
this issue is to decide the location and order of 
the leading pole for $Z_f(\varphi)$.  
In the seminal work of Varchenko \cite{Var76},
when $f$ is real analytic and satisfies some conditions,  
$c_0(f)$ can be expressed by using 
the {\it Newton polyhedron} of $f$ as
\begin{equation}\label{eqn:1.6}
c_0(f)=1/d(f),
\end{equation}
where $d(f)$ is the {\it Newton distance} of $f$
(see \cite{Var76}, \cite{AGV88}) and the order of pole
at $s=-1/d(f)$ depends on some topological 
information of the Newton polyhedron of $f$.
He uses the theory of toric varieties 
based on the geometry of Newton polyhedra.
More detailed situation of meromorphic continuation 
of $Z_f(\varphi)$ is investigated in 
\cite{DS89}, \cite{DS92}, \cite{DNS05}, \cite{OkT13}, etc. 
A recent interesting work \cite{CGP13} treating the equality 
(\ref{eqn:1.6}) is from another approach.
In the same paper \cite{Var76}, 
Varchenko more deeply investigated the two-dimensional case. 
Indeed, without any assumption, he shows that 
the equality (\ref{eqn:1.6}) holds for real analytic $f$
on {\it adapted coordinates}. 
Here adapted coordinates are important coordinates 
in the study of oscillatory integrals 
and their existence is shown in two-dimensions 
in \cite{Var76}, \cite{PSS99}, \cite{IkM11tams}, etc.
More generally, let us consider the smooth case. 
In the above cited paper \cite{KaN16jmsut},
the authors show that Varchenko's result can be naturally
generalized in a certain restricted class of smooth functions. 
On the other hand, 
M. Greenblatt \cite{Gre06} obtains a sharp result 
which generalizes the Varchenko's two-dimensional result.
\begin{theorem}[Greenblatt \cite{Gre06}]
When $f$ is a smooth function defined on $U$ in $\R^2$,  
the equality (\ref{eqn:1.6}) holds on adapted coordinates. 
\end{theorem}
In more detail, in his same paper \cite{Gre06}, 
Greenblatt explains the delicate situation
about the local integrability of $|f|^{-\mu}$ 
around $\mu=c_0(f)$ in the smooth case
by using the specific function:
\begin{equation}\label{eqn:1.7}
f(x,y)=x^a y^b+x^a y^{b-2} e^{-1/|x|^{1/(2b)}},
\end{equation}
where $a,b$ are nonnegative integers satisfying 
$a<b$ and $2\leq b$ 
(see Remark~3.2 in this paper for his result).
Note that the second term in (\ref{eqn:1.7}) is a 
(non-real analytic) flat function. 
The purpose of this paper is to investigate 
a slight generalization of the above example more deeply 
and to understand detailed situation 
of analytic continuation of the 
respective local zeta functions.

This paper is organized as follows.
In Section~2, 
we state a main theorem showing the failure
of meromorphy of some local zeta functions. 
In order to show the main theorem, 
we substantially investigate similar integrals $Z(\sigma)$ 
in Section~3.
The main theorem is shown in Section~4
by using the results in Section~3. 
In Section~5, we give some property of domains of 
convergence of the integrals $Z_f(\varphi)(s)$,
which is analogous to Landau's theorem on the Dirichlet series 
with positive coefficients.
This result implies Lemma~1.1. 
Our computations in this paper are very specific 
and it is hoped to give good observation for future studies about 
properties of local zeta functions in the general smooth case. 
From our results, 
many elementary (but probably not so easy) questions
are naturally raised, 
some of which are listed in Section~6.

In this paper, we use  
$C, C_1, C_2, \epsilon, \delta$ for
various kinds of constants without further comments.


\section{Main results}

In this section, we consider 
the integrals $Z_f(\varphi)(s)$
with
smooth functions $f$ of the following form:
\begin{equation}\label{eqn:2.1}
f(x,y)=x^a y^b+x^a y^{b-q} e^{-1/|x|^p},
\end{equation}
where
$a,b,p,q$ satisfy 
\begin{itemize}
\item $a,b,q$ are nonnegative integers satisfying 
$a<b$, $2\leq b$, $1\leq q\leq b$;
\item $p$ is a positive real number.
\end{itemize}

We remark that $e^{-1/|x|^p}$ is regarded as 
a smooth function defined on $\R$ 
by considering that its value at $0$ takes $0$. 
As mentioned in the Introduction, 
the above function $f$ slightly 
generalizes the function (\ref{eqn:1.7}) which 
is investigated by Greenblatt \cite{Gre06}
(he considers the case when $q=2$ and $p=1/(2b)$).
Note that the coordinate $(x,y)$ satisfies 
the adapted coordinates conditions in \cite{IkM11tams}  
and that the Newton distance 
of $f$ is $b$ (i.e., $d(f)=b$). 

Let us consider the case when 
the second term, which is flat, 
does not appear in (\ref{eqn:2.1})
(i.e., $f(x,y)=x^a y^b$).
From \cite{GeS64}, it is easy to see that
$Z_f(\varphi)$ can be regarded as a meromorphic function
on $\C$ and the poles of $Z_f(\varphi)$ are contained
in the set $\{-j/a,-k/b:j,k\in\N\}$. 
When $\varphi$ satisfies (\ref{eqn:1.4}), 
the leading pole exists
at $s=-1/b$, whose order is one.

Now, when $f$ is as in (\ref{eqn:2.1}), 
it follows from
Theorem~1.2 and Lemma~1.1 in the Introduction 
that $Z_f(\varphi)$ is holomorphic on 
the region: ${\rm Re}(s)>-1/b$ and 
that $Z_f(\varphi)$
has a singularity at $s=-1/b$. 
(Note that $c_0(f)=1/b$.) 
More precisely, 
we see the behavior at $-1/b$ of the restriction of
$Z_f(\varphi)$ to the real axis as follows. 
In this paper, we use the symbol $s=\sigma+i t$
with $\sigma, t\in\R$ which is 
traditionally used in the analysis of the Riemann 
zeta function. 
\begin{theorem}\label{thm:1.1}
Let $f$ be as in \eqref{eqn:2.1}
and $\varphi$ as in \eqref{eqn:1.1}.
We assume that $q$ is even.
Then
the following hold:
\begin{enumerate}
\item
If $p>1-a/b$, 
then
\begin{equation}\label{eqn:2.2}
\lim_{\sigma\to-1/b+0}
(b\sigma+1)^{1-\frac{1-a/b}{p}}\cdot Z_f(\varphi)(\sigma)
=4A\cdot\varphi(0,0)
\end{equation}
where $A$ is the positive constant defined by 
\begin{equation}\label{eqn:A}
A=\int_0^{\infty}x^{-a/b}(1-e^{-1/(qx^p)})dx.
\end{equation}
Note that the above improper integral converges. 
\item
If $p=1-a/b$, 
then 
\begin{equation}\label{eqn:2.3}
\lim_{\sigma\to-1/b+0}
|\log (b\sigma+1)|^{-1}\cdot Z_f(\varphi)(\sigma)
=\frac{4}{pq}\cdot\varphi(0,0).
\end{equation}
\item
If $0<p< 1-a/b$, then
there exists a constant $B(\varphi)$, 
which depends on $a,b,p,q,\varphi$ but is independent
of $\sigma$, such that 
\begin{equation}\label{eqn:2.4}
\lim_{\sigma\to-1/b+0}Z_f(\varphi)(\sigma)=B(\varphi),
\end{equation}
where $B(\varphi)$ is positive 
if $\varphi$ satisfies the condition (\ref{eqn:1.4}). 
\end{enumerate}
\end{theorem}
Of course, if $Z_f(\varphi)$ had a pole of order $m$ at $s=-1/b$, 
then 
$
\lim_{\sigma\to-1/b+0}
(b\sigma+1)^{m} \cdot Z_f(\varphi)(\sigma)
$ must be a positive value. 
Noticing that $0<1-\frac{1-a/b}{p}<1$, 
we can see the following from the above theorem
with Lemma~1.1 in the Introduction. 
\begin{corollary}
Under the assumption in Theorem~2.1 
with the condition (\ref{eqn:1.4}) on $\varphi$,
$Z_f(\varphi)$ cannot be meromorphically 
continued to any open neighborhood of $s=-1/b$ in $\C$.
In other words, 
the singularity of $Z_f(\varphi)(s)$ at $s=-1/b$
is different from a pole.
\end{corollary}

If ${\rm Re}(s)>-c_0(f)$, then 
$|f|^s$ can be regarded as a {\it distribution}
by considering the map from
$C_0^{\infty}(U)$ to $\C$ defined by 
\begin{equation*}
\varphi
\longmapsto
\left\langle
|f|^s,\varphi
\right\rangle=
\int_{\R^2}|f|^s\varphi dxdy=Z_f(\varphi)(s).
\end{equation*} 
Furthermore, the equalities (\ref{eqn:2.2}), (\ref{eqn:2.3}) 
in (i), (ii) can be interpreted as 
in the following.   
\begin{equation}\label{eqn:2.5}
\begin{split}
&\lim_{\sigma\to-1/b+0}
(b\sigma+1)^{1-\frac{1-a/b}{p}}|f|^{\sigma}=4A\delta,\\
&\lim_{\sigma\to-1/b+0}
|\log(b\sigma+1)|^{-1}|f|^{\sigma}=\frac{1}{pq}\delta,\\
\end{split}
\end{equation}
where $\delta\in{\mathcal D}'(U)$ is Dirac's delta function. 
The limits in the left-hand sides of (\ref{eqn:2.5}) are 
taken in the topology of ${\mathcal D}'(U)$.
On the other hand, 
the map $B$ from $C_0^{\infty}(U)$ to $\R$, 
defined by 
$\varphi\mapsto B(\varphi)
(=\left\langle B,\varphi \right\rangle)$ from (iii), 
can also be interpreted as a distribution: 
\begin{equation*}
\lim_{\sigma\to-1/b+0} |f|^{\sigma}=B.
\end{equation*}
The continuity of the above map $B$ will be explained 
in Remark~4.1.

\section{Asymptotic limits of associated integrals}

For a set $U$ in $\R^2$, 
let us define the integral of the form:
\begin{equation}\label{eqn:3.1}
\begin{split}
Z_U(\sigma):=\int_{U}
\left|x^ay^b+x^ay^{b-q}e^{-1/|x|^p}\right|^{\sigma}
dxdy
\quad\quad \mbox{ for $\sigma<0$,}
\end{split}
\end{equation}
where $a,b,p,q$ satisfy the conditions in the previous section. 

In this section, we consider the case when 
$U=U_+(r_1,r_2)$ with $r_1,r_2\in(0,1)$, 
where 
$$
U_+(r_1,r_2)=
\{(x,y)\in\R^2:0\leq x\leq r_1, 0\leq y\leq r_2\}
$$ 
and simply denote 
\begin{equation}\label{eqn:3.2}
Z(\sigma):=Z_{U_+(r_1,r_2)}(\sigma).
\end{equation} 
%
Let $e(x)$ be the smooth function defined by 
\begin{equation}\label{eqn:3.3}
e(x):=\exp\left(\frac{-1}{qx^p}\right) 
\quad\quad \mbox{ for $x>0$}
\end{equation}
and $e(0)=0$, 
which frequently appears in the computation below.
Since the function $e$ is monotonously increasing,
we can define the function $\rho:[0,\infty)\to[0,r_1]$ by
\begin{equation}\label{eqn:rho}
\rho(y)=
\begin{cases}
e^{-1}(y)=\left(
\frac{-1}{q\log y}
\right)^{1/p} \quad\quad \mbox{ if $0\leq y <e(r_1)$,} \\
r_1 \quad\quad\quad\quad \mbox{ if $y\geq e(r_1)$.}
\end{cases}
\end{equation}

Since the Newton distance of $f$ equals $b$, 
Theorem~1.2 due to Greenblatt implies that  
the integral $Z(\sigma)$ converges if $\sigma>-1/b$ 
and diverges if $\sigma<-1/b$.
The purpose of this section is 
to compute exact asymptotic limits 
of $Z(\sigma)$ as $\sigma\to-1/b+0$.

\begin{theorem}\label{thm:5.1}
The integral $Z(\sigma)$ satisfies the following.
\begin{enumerate}
\item
If $p>1-a/b$, 
then
\begin{equation*}
\lim_{\sigma\to-1/b+0}
(b\sigma+1)^{1-\frac{1-a/b}{p}}\cdot Z(\sigma)
=A, 
\end{equation*}
where $A$ is as in (\ref{eqn:A}).
\item
If $p=1-a/b$, 
then 
\begin{equation*}
\lim_{\sigma\to-1/b+0}
|\log (b\sigma+1)|^{-1}\cdot Z(\sigma)
=\frac{1}{pq}.
\end{equation*}
\item
If $0<p< 1-a/b$, then the limit of 
$Z(\sigma)$ as 
$\sigma\to-1/b+0$ exists 
and it satisfies 
\begin{equation*}
\begin{split}
&\max\left\{
\frac{L(\lambda)}{(1+\lambda^q)^{1/b}}+
\frac{M(\lambda)}{(1+\lambda^{-q})^{1/b}}
:\lambda>0
\right\} \\
&\quad\quad\leq 
\lim_{\sigma\to-1/b+0}Z(\sigma) \leq 
\min\{L(\lambda)+M(\lambda):\lambda>0\},
\end{split}
\end{equation*}
where $L(\lambda)$, $M(\lambda)$ are 
positive constants depending on $\lambda$ as in 
(\ref{eqn:3.12}), (\ref{eqn:3.38}) in the proof below.
\end{enumerate}
\end{theorem}

\begin{remark}
In \cite{Gre06}, 
Greenblatt shows the
boundedness of $Z(\sigma)$  
near $\sigma=-1/b$ 
in the case of $p=1/(2b)$ and $q=2$, 
which is contained in the above case (iii).
\end{remark}


\subsection{Auxiliary function with a parameter}

Let $\lambda$ be a positive number. 
The set $U_+(r_1,r_2)$ is decomposed as
$U_1(\lambda)\cup U_2(\lambda)$ with
\begin{equation*}
\begin{split}
&U_1(\lambda)=\{(x,y)\in U_+(r_1,r_2): 
\lambda y\geq e(x)\},\\
&U_2(\lambda)=\{(x,y)\in U_+(r_1,r_2):
\lambda y < e(x))\}.
\end{split}
\end{equation*}
The integral $Z(\sigma)$ is expressed as 
\begin{equation}\label{eqn:3.5}
Z(\sigma)=
Z_1^{(\lambda)}(\sigma)
+
Z_2^{(\lambda)}(\sigma),
\end{equation}
where 
\begin{equation}\label{eqn:3.6}
\begin{split}
Z_j^{(\lambda)}(\sigma)
=\int_{U_j(\lambda)}
(x^a y^b +x^a y^{b-q} e^{-1/x^p})^{\sigma}dxdy 
\quad\quad \mbox{for $j=1,2$.}
\end{split}
\end{equation}
Note that 
\begin{equation}\label{eqn:3.7}
\begin{split}
&x^a y^b(1 + y^{-q} e^{-1/x^p})\leq
(1+\lambda^q) x^a y^b \quad \mbox{for $(x,y)\in U_1(\lambda)$,} \\
&x^a y^{b-q} e^{-1/x^p}(y^q e^{1/x^p}+1)\leq 
(1+\lambda^{-q}) x^a y^{b-q} e^{-1/x^p} 
\quad \mbox{for $(x,y)\in U_2(\lambda)$.}
\end{split}
\end{equation}
Thus each $Z_j^{(\lambda)}(\sigma)$ can be estimated 
by using the following integrals.
\begin{eqnarray}
&&\tilde{Z}_1^{(\lambda)}(\sigma)=
\int_{U_1(\lambda)}x^{a\sigma}y^{b\sigma}dxdy,
\label{eqn:3.8}\\
&&\tilde{Z}_2^{(\lambda)}(\sigma)=
\int_{U_2(\lambda)}
x^{a\sigma}y^{(b-q)\sigma}e^{-\sigma/x^p}dxdy.
\label{eqn:3.9}
\end{eqnarray}
Indeed, 
${Z}_1^{(\lambda)}(\sigma)$ and 
${Z}_2^{(\lambda)}(\sigma)$ are convergent 
if and only if  
so are 
$\tilde{Z}_1^{(\lambda)}(\sigma)$ and 
$\tilde{Z}_2^{(\lambda)}(\sigma)$.
Moreover, they satisfy 
\begin{equation}\label{eqn:3.10}
\begin{split}
&(1+\lambda^q)^{\sigma}
\tilde{Z}_1^{(\lambda)}(\sigma)< 
Z_1^{(\lambda)}(\sigma)< 
\tilde{Z}_1^{(\lambda)}(\sigma),\\
&(1+\lambda^{-q})^{\sigma}
\tilde{Z}_2^{(\lambda)}(\sigma)< 
Z_2^{(\lambda)}(\sigma)
< \tilde{Z}_2^{(\lambda)}(\sigma),
\end{split}
\end{equation}
for $\sigma<0$.
In order to prove Theorem~3.1, 
let us investigate the behaviors of 
$\tilde{Z}_1^{(\lambda)}(\sigma)$ and 
$\tilde{Z}_2^{(\lambda)}(\sigma)$ 
as $\sigma\to-1/b+0$.

\subsection{Preliminary lemma}

For $\alpha>0$, 
let $\psi_{\alpha}$ be the smooth function 
defined by 
\begin{equation}\label{eqn:3.11}
\psi_{\alpha}(x)
:=\frac{1-\alpha^x}{x} \quad \mbox{ for $x>0$.} 
\end{equation} 
The following properties of $\psi_{\alpha}$ play important roles 
in the computation below.

\begin{lemma} 
The function $\psi_{\alpha}$ satisfies the following 
properties.
\begin{enumerate}
\item There exist positive constants $C_1, C_2$ such that  
\begin{equation*}
-\log\alpha-C_1x<\psi_{\alpha}(x)
<-\log\alpha -C_2x
\end{equation*}
for $x\in(0,1)$.
In particular, $\lim_{x\to +0}\psi_{\alpha}(x)=-\log\alpha$.
\item If $\alpha\in(0,1)$, 
then $\lim_{x\to \infty}\psi_{\alpha}(x)=0$.
\item If $\alpha\in(0,1)$, 
then $\psi_{\alpha}$ is monotonously decreasing, 
in particular, 
$0<\psi_{\alpha}(x)<-\log \alpha$ for $x>0$.
\end{enumerate}
\end{lemma}

\begin{proof} 
The above properties of $\psi_{\alpha}$ can be 
easily seen by using Taylor's formula.
\end{proof}

\begin{remark}
In the computation below, 
the function $1-e(x)$ often appears. 
This function can be expressed by using $\psi_{\alpha}$
with $\alpha=\exp(-q^{-1})$ as follows. 
\begin{equation*}
\begin{split}
1-e(x)=
\frac{1-\exp(-q^{-1}x^{-p})}{x^{-p}}\cdot x^{-p} 
=\psi_{\alpha}(x^{-p})x^{-p}.
\end{split}
\end{equation*}
From Lemma~3.3, we can see that 
\begin{equation}\label{eqn:e2}
q^{-1}x^{-p}-C_1x^{-2p}\leq
1-e(x)\leq
q^{-1}x^{-p}-C_2x^{-2p} \mbox{ for $x\geq 1$,}
\end{equation}
where $C_1,C_2$ are positive constants. 
\end{remark}

\subsection{Asymptotics of $\tilde{Z}_1^{(\lambda)}(\sigma)$}

Let us investigate 
the behavior of $\tilde{Z}_1^{(\lambda)}(\sigma)$ 
which is essentially important.

\begin{lemma}
\begin{enumerate}
\item
If $p>1-a/b$, 
then
\begin{equation*}
\lim_{\sigma\to-1/b+0}(b\sigma+1)^{1-\frac{1-a/b}{p}}\cdot 
\tilde{Z}_1^{(\lambda)}(\sigma)=A,
\end{equation*}
where $A$ is as in (\ref{eqn:A}).
\item
If $p=1-a/b$, 
then
\begin{equation*}
\lim_{\sigma\to-1/b+0}
|\log (b\sigma+1)|^{-1}\cdot \tilde{Z}_1^{(\lambda)}(\sigma)
=\frac{1}{pq}.
\end{equation*} 
\item
If $0<p<1-a/b$, then
$\lim_{\sigma\to-1/b+0}
\tilde{Z}_1^{(\lambda)}(\sigma)=L(\lambda)$ with
\begin{equation}\label{eqn:3.12}
L(\lambda)=
\frac{\rho(\lambda r_2)^{1-a/b}}{1-a/b}
\cdot\log(\lambda r_2)
+
\frac{\rho(\lambda r_2)^{1-a/b-p}}{q(1-a/b-p)}.
\end{equation}
\end{enumerate}
\end{lemma}

\begin{proof}
In the proof, we introduce the variable: 
\begin{equation*}
X=b\sigma+1,
\end{equation*}
which is convenient for many kinds of limit processes later. 
Note that 
$$\sigma\to -1/b+0 \Longleftrightarrow X\to +0.$$

Now, applying an iterated integral to (\ref{eqn:3.8}), 
we have
\begin{equation}\label{eqn:3.13}
\tilde{Z}_1^{(\lambda)}\left(\sigma\right)
=\frac{1}{\lambda^X}
\frac{1}{X}\int_0^{\rho(\tilde{r}_2)} 
x^{a\sigma} (\tilde{r}_2^X-e(x)^X)dx,
\end{equation}
where $\tilde{r}_2:=\lambda r_2$.



\vspace{.5 em}

{\bf [The case (i) : $p>1-a/b$.]} \quad 

\vspace{.2 em}

Changing the integral variable in (\ref{eqn:3.13}) by  
\begin{equation*}
x=X^{1/p} u
\Longleftrightarrow 
u=X^{-1/p}x, \quad (dx=X^{1/p}du), 
\end{equation*}
we have 
\begin{equation}\label{eqn:3.14}
\tilde{Z}_1^{(\lambda)}(\sigma)=
\frac{1}{\lambda^X}
\frac{1}{X^{1-\frac{1+a\sigma}{p}}}
\int_0^{\rho(\tilde{r}_2)/X^{1/p}} u^{a\sigma}(\tilde{r}_2^X-e(u))du.
\end{equation}
We decompose the integral in (\ref{eqn:3.14}) 
as $G_1(\sigma)+G_2(\sigma)+G_3(\sigma)$
with 
\begin{eqnarray}
&&G_1(\sigma)=
\int_0^{1} u^{a\sigma}(\tilde{r}_2^X-e(u))du, 
\nonumber\\
&&G_2(\sigma)=
\int_1^{\rho(\tilde{r}_2)/X^{1/p}} 
u^{a\sigma}(1-e(u))du,
\label{eqn:3.15} \\
&&G_3(\sigma)=
(\tilde{r}_2^X-1)
\int_1^{\rho(\tilde{r}_2)/X^{1/p}} 
u^{a\sigma}du.
\nonumber
\end{eqnarray}

\vspace{.2 em}
\underline{The limit of $G_1(\sigma)$.}

\vspace{.2 em}

Since $\tilde{r}_2^X\leq \max\{1,\lambda\}$, 
the integrand can be estimated as 
\begin{equation*}
u^{a\sigma}(\tilde{r}_2^X-e(u))< u^{-a/b} \max\{1,\lambda\}
\end{equation*}
for $(X,u)\in (0,1]\times(0,1]$.
Since $-a/b>-1$, 
the Lebesgue convergence theorem implies 
\begin{equation}\label{eqn:3.16}
\lim_{\sigma\to-1/b+0} G_1(\sigma)
=\int_0^{1}u^{-a/b}(1-e(u))du.
\end{equation}

\vspace{.2 em}
\underline{The limit of $G_2(\sigma)$.}

\vspace{.2 em}

Let $\epsilon_0:=
a/b+p-1>0$, then 
\begin{equation}\label{eqn:3.20}
a\sigma-p+1=\frac{a}{b}(b\sigma+1)-\frac{a}{b}-p+1
=\frac{a}{b}X-\epsilon_0.
\end{equation}
Since $u<\rho(\tilde{r}_2)/X^{1/p}$ 
($\Leftrightarrow X<\rho(\tilde{r}_2)^p u^{-p}$)
and $e^{x}\leq 1+(e-1)x$ for $x\in(0,1)$, 
there exists $\delta>0$ such that 
\begin{equation}\label{eqn:3.23}
\begin{split}
1<u^{\frac{a}{b}X}=
\exp\left(\frac{a}{b}X\cdot \log u\right)
< 1+ D u^{-p}\log u,
\end{split}
\end{equation}
for 
$(u,X)\in [1,\rho(\tilde{r}_2)/X^{1/p})\times(0,\delta)$,
where $D:=(e-1)\frac{a}{b}\rho(\tilde{r}_2)^p$.
From (\ref{eqn:3.20}), (\ref{eqn:e2}), (\ref{eqn:3.23}), 
we have
\begin{equation*}
\begin{split}
u^{a\sigma}(1-e(u))
&< q^{-1}u^{a\sigma-p}
= q^{-1}u^{\frac{a}{b}X-1-\epsilon_0} \\
&
< q^{-1}(1+Du^{-p}\log u)u^{-1-\epsilon_0}
\leq Cu^{-1-\epsilon_0}
\end{split}
\end{equation*}
for $(u,X)\in [1,\rho(\tilde{r}_2)/X^{1/p})\times(0,\delta)$,
where $C>0$ is a constant independent of $u$ and $\sigma$.
Thus, the Lebesgue convergence theorem implies  
\begin{equation}\label{eqn:3.17}
\begin{split}
\lim_{\sigma\to-1/b+0} 
G_2(\sigma)
=\int_1^{\infty}u^{-a/b}(1-e(u))du.
\end{split}
\end{equation}

\vspace{.2 em}
\underline{The limit of $G_3(\sigma)$.}

\vspace{.2 em}

Since $1-\tilde{r}_2^X=X\psi_{\tilde{r}_2}(X)$, 
$G_3(\sigma)$ can be computed as follows.
\begin{equation*}
\begin{split}
G_3(\sigma) 
=X^{1-\frac{a\sigma+1}{p}}
\psi_{\tilde{r}_2}(X)\cdot
\frac{\rho(\tilde{r}_2)^{a\sigma+1}-X^{(a\sigma+1)/p}}{a\sigma+1}.
\end{split}
\end{equation*}
From Lemma~3.3 (i), 
we can see the following. 
\begin{equation}\label{eqn:3.18}
\begin{split}
\lim_{\sigma\to -1/b} X^{-1+\frac{a\sigma+1}{p}}
G_3(\sigma)=
(-\log \tilde{r}_2) \cdot 
\frac{\rho(\tilde{r}_2)^{1-a/b}}{1-a/b}. 
\end{split}
\end{equation}

\vspace{.2 em}

Therefore, (\ref{eqn:3.16}), (\ref{eqn:3.17}), (\ref{eqn:3.18}) 
imply  
\begin{equation*}
\begin{split}
\lim_{\sigma\to-1/b+0} 
X^{1-\frac{1-a/b}{p}}
\tilde{Z}_1^{(\lambda)}(\sigma) 
=
\lim_{\sigma\to-1/b+0} G_1(\sigma)+
\lim_{\sigma\to-1/b+0}G_2(\sigma)=A.
\end{split}
\end{equation*}
Here we used the fact: 
$\lim_{X\to +0} X^{cX}=1$ ($c\in\R$).

\vspace{.5 em}
{\bf [The case (ii) : $p=1-a/b$.]} \quad 

Since
the equalities (\ref{eqn:3.16}) and (\ref{eqn:3.18})
always hold for any $p>0$,   
it suffices to consider the behavior of $G_2(\sigma)$
in the case of $p=1-a/b$.

\vspace{.2 em}

\underline{The limit of $G_2(\sigma)$.}

\vspace{.2 em}

Since $\epsilon_0=0$ in (\ref{eqn:3.20}),
$\a\sigma=p-1+\frac{a}{b}X$ holds. 
Thus, the estimates
(\ref{eqn:e2}), (\ref{eqn:3.23}) imply that
there exist $C_1, C_2,\delta>0$ such that 
\begin{equation}\label{eqn:e3}
q^{-1}u^{-1}-C_1 u^{-p-1}<
u^{a\sigma}(1-e(u))<
q^{-1}u^{-1}+C_2 u^{-p-1}\log u,
\end{equation}
for $(u,X)\in 
[1,\rho(\tilde{r}_2)/X^{1/p})\times(0,\delta)$.


Now, we rewrite (\ref{eqn:3.15}) as follows.
\begin{equation*}
\begin{split}
G_2(\sigma)&=
q^{-1}\int_1^{\rho(\tilde{r}_2)/X^{1/p}} \frac{1}{u}du-
\int_1^{\rho(\tilde{r}_2)/X^{1/p}} 
\left(\frac{q^{-1}}{u}-u^{a\sigma}(1-e(u))\right)du \\
&=:H_1(\sigma)-H_2(\sigma).
\end{split}
\end{equation*}
A direct computation gives 
\begin{equation}\label{eqn:3.24}
H_1(\sigma)=q^{-1}(\log\rho(\tilde{r}_2)-p^{-1}\log X).
\end{equation}
From (\ref{eqn:e3}), 
there exist $\epsilon, C>0$ 
such that 
\begin{equation}\label{eqn:3.25}
|H_2(\sigma)|
\leq C \int_1^{\rho(\tilde{r}_2)/X^{1/p}} 
u^{-p+\epsilon-1}du
\leq \frac{C}{p-\epsilon}.
\end{equation}
From 
(\ref{eqn:3.16}), (\ref{eqn:3.18}), 
(\ref{eqn:3.24}), (\ref{eqn:3.25}),
we have 
\begin{equation*}
\begin{split}
&\lim_{\sigma\to -1/b}
|\log X|^{-1} \tilde{Z}_1^{(\lambda)}(\sigma)
=
\lim_{\sigma\to -1/b}
|\log X|^{-1} G_2(\sigma)\\
&\quad\quad
=
\lim_{\sigma\to -1/b}
|\log X|^{-1} H_1(\sigma)=\frac{1}{pq}.
\end{split}
\end{equation*}

 \vspace{.5 em}
{\bf [The case (iii) : $0<p<1-a/b$.]} \quad 

\vspace{.2 em}

From (\ref{eqn:3.13}), 
$\tilde{Z}_1^{(\lambda)}(\sigma)$ 
can be decomposed as 
$J_1(\sigma)+J_2(\sigma)$ with 

\begin{equation*}
\begin{split}
&J_1(\sigma)=
\frac{1}{\lambda^X}
\int_0^{\rho(\tilde{r}_2)} 
x^{a\sigma} \frac{1-e(x)^X}{X}dx, \\
&J_2(\sigma)=
\frac{1}{\lambda^X}\cdot
\frac{\tilde{r}_2^X-1}{X}
\int_0^{\rho(\tilde{r}_2)} 
x^{a\sigma}dx.
\end{split}
\end{equation*}

\vspace{.2 em}

\underline{The limit of $J_1(\sigma)$.}

\vspace{.2 em}

Using the function $\psi_{\alpha}$ with
$\alpha:=\exp(-q^{-1})$ and Lemma~3.3 (i), 
we have
\begin{equation}\label{eqn:3.26}
\begin{split}
\frac{1-e(x)^{X}}{X} 
=\frac{\psi_{\alpha}(X x^{-p})}{x^p} \longrightarrow 
q^{-1}x^{-p}
\quad \mbox{ as $\sigma\to-1/b+0$.}
\end{split}
\end{equation}
On the other hand, 
Lemma~3.3 (iii) implies  
\begin{equation*}
\left|
x^{a\sigma}\cdot\frac{1-e(x)^{X}}{X}
\right|
\leq C x^{-a/b-p}
\end{equation*}
for $(x,X)\in(0,\rho(\tilde{r}_2)]\times(0,1)$, 
where $C>0$ is a constant depending only 
on $\tilde{r}_2$.
Since $-a/b-p>-1$, 
the Lebesgue convergence theorem 
implies
\begin{equation}\label{eqn:3.27}
\begin{split}
\lim_{\sigma\to -1/b+0} 
J_1(\sigma)=
\frac{1}{q}
\int_0^{\rho(\tilde{r}_2)} x^{-a/b-p}dx
=
\frac{\rho(\tilde{r}_2)^{1-a/b-p}}{q(1-a/b-p)}.
\end{split}
\end{equation}


\vspace{.2 em}

\underline{The limit of $J_2(\sigma)$.}

\vspace{.2 em}

A direct computation gives
$$
J_2(\sigma)=
-\frac{1}{\lambda^X}\cdot
\psi_{\tilde{r}_2}(X)\cdot
\frac{\rho(\tilde{r}_2)^{a\sigma+1}}{a\sigma+1}.
$$
Therefore we have 
\begin{equation}\label{eqn:3.28}
\lim_{\sigma\to -1/b+0} 
J_2(\sigma)=
\frac{\rho(\tilde{r}_2)^{1-a/b}}{1-a/b}
\cdot\log(\tilde{r}_2)
\end{equation}

From (\ref{eqn:3.27}), (\ref{eqn:3.28}), 
we obtain the limit 
of $\tilde{Z}_1^{(\lambda)}(\sigma)$ in (iii).

\end{proof}

\subsection{Asymptotics of $\tilde{Z}_2^{(\lambda)}(\sigma)$}

The behavior of $\tilde{Z}_2^{(\lambda)}(\sigma)$ 
can be easily seen by 
a direct computation. 

\begin{lemma}
$\lim_{\sigma\to -1/b}\tilde{Z}_2^{(\lambda)}(\sigma)=M(\lambda)$ 
with
\begin{equation}\label{eqn:3.38}
\begin{split}
M(\lambda)=
\frac{b^2}{q(b-a)}\cdot \frac{1}{\lambda^{q/b}}\cdot 
\rho(\lambda r_2)^{1-a/b} 
+\frac{b}{q}r_2^{q/b}\int_{\rho(\lambda r_2)}^{r_1} 
x^{-a/b}\exp(1/(bx^p))dx.
\end{split}
\end{equation}
\end{lemma}

\begin{proof}
By decomposing the integral region 
$U_2(\lambda)$ into the following two sets:
$$
\{(x,y)\in U_2(\lambda):x\leq \rho(\tilde{r}_2)\},\quad
\{(x,y)\in U_2(\lambda):x > \rho(\tilde{r}_2)\},
$$
the integral $\tilde{Z}_2^{(\lambda)}(\sigma)$ can be computed as
\begin{equation*}
\tilde{Z}_2^{(\lambda)}(\sigma)
=\frac{1}{X-q\sigma}\cdot\frac{1}{\lambda^{X-q\sigma}}
\int_0^{\rho(\tilde{r}_2)} x^{a\sigma} e(x)^{X}dx+
\frac{r_2^{X-q\sigma}}{X-q\sigma}
\int_{\rho(\tilde{r}_2)}^{r_1} x^{a\sigma}e^{-\sigma/x^p}dx.
\end{equation*} 
It can be easily computed that 
the limit of the right-hand side of the above equation as 
$\sigma\to-1/b+0$ is $M(\lambda)$ 
in (\ref{eqn:3.38}).   
\end{proof}

\subsection{Proof of Theorem~3.1}

From (\ref{eqn:3.5}), (\ref{eqn:3.10}), 
when the integrals 
$\tilde{Z}_1^{(\lambda)}(\sigma)$ and 
$\tilde{Z}_2^{(\lambda)}(\sigma)$ 
converge, $Z(\sigma)$ can be estimated as
\begin{equation}\label{eqn:3.39}
\begin{split}
&(1+\lambda^q)^{\sigma}\cdot \tilde{Z}_1^{(\lambda)}(\sigma)+
(1+\lambda^{-q})^{\sigma}\cdot \tilde{Z}_2^{(\lambda)}(\sigma) \\
&\quad\quad\quad\quad\quad\quad\quad\quad <Z(\sigma)< 
\tilde{Z}_1^{(\lambda)}(\sigma)+
\tilde{Z}_2^{(\lambda)}(\sigma).
\end{split}
\end{equation}

First, let us consider the case (i). 
Since Lemmas~3.5 (i) and 3.7 imply 
$$
\lim_{\sigma\to-1/b+0}X^{1-\frac{1-a/b}{p}}\cdot
\tilde{Z}_1^{(\lambda)}(\sigma)=A, 
\quad
\lim_{\sigma\to-1/b+0}X^{1-\frac{1-a/b}{p}}\cdot
\tilde{Z}_2^{(\lambda)}(\sigma)=0,
$$
the estimates (\ref{eqn:3.39}) give
\begin{equation}\label{eqn:3.40}
\begin{split}
(1+\lambda^q)^{-1/b}\cdot A
&\leq
\varliminf_{\sigma\to-1/b+0}X^{1-\frac{1-a/b}{p}}\cdot Z(\sigma)\\
&\leq
\varlimsup_{\sigma\to-1/b+0}X^{1-\frac{1-a/b}{p}}\cdot Z(\sigma) 
\leq A.
\end{split}
\end{equation}
Note that $Z(\sigma)$ is independent of $\lambda$.
Considering the limit  as $\lambda\to 0$ in (\ref{eqn:3.40}), 
we obtain (i). 
The case (ii) can be similarly shown. 

Let us consider the case (iii).
Since the integral $Z(\sigma)$ is a monotone decreasing function
as $\sigma\in(-1/b,0)$, the boundedness 
of $Z(\sigma)$ from (\ref{eqn:3.39}) implies 
the existence of the limit $\lim_{\sigma\to-1/b}Z(\sigma)$. 
Moreover, since $Z(\sigma)$ is independent of $\lambda$,   
the following inequalities can be obtained
from (\ref{eqn:3.39}), (\ref{eqn:3.12}), (\ref{eqn:3.38}). 
\begin{equation}\label{eqn:3.41}
\begin{split}
&\sup\left\{
\frac{L(\lambda)}{(1+\lambda^q)^{1/b}}+
\frac{M(\lambda)}{(1+\lambda^{-q})^{1/b}}
:\lambda>0
\right\} \\
&\quad\quad\leq 
\lim_{\sigma\to-1/b+0}Z(\sigma) \leq 
\inf\{L(\lambda)+M(\lambda):\lambda>0\}.
\end{split}
\end{equation}
Furthermore, the supremum and infimum in (\ref{eqn:3.41}) can be
respectively replaced by the maximum and minimum 
by using the lemma below. 
As a result, the inequalities in (iii) 
in the theorem is obtained.

\begin{lemma}
\begin{equation*}
\begin{split}
&\lim_{\lambda\to 0}L(\lambda)
=0, \quad 
\lim_{\lambda\to 0}M(\lambda)
=\infty, \quad 
\lim_{\lambda\to \infty}L(\lambda)
=\infty, \quad 
\lim_{\lambda\to \infty}M(\lambda)
=0, \\
&\lim_{\lambda\to 0}
(1+\lambda^{q})^{-1/b}\cdot L(\lambda)=0, \quad 
\lim_{\lambda\to 0}
(1+\lambda^{-q})^{-1/b}\cdot M(\lambda)=0, \\
&\lim_{\lambda\to \infty}
(1+\lambda^{q})^{-1/b}\cdot L(\lambda)=0, \quad  
\lim_{\lambda\to \infty}
(1+\lambda^{-q})^{-1/b}\cdot M(\lambda)=0.
\end{split}
\end{equation*}
\end{lemma}
The proof of the above lemma is easy, 
so it is left to the readers. 



\section{Proof of Theorem~2.1}

For $r_1,r_2\in(0,1)$, let 
$$
U(r_1,r_2):=\{(x,y)\in\R^n:|x|<r_1,|y|<r_2\}.
$$
The behavior of $Z_f(\varphi)$ can be appropriately
approximated by that of a more simple function 
$Z_{U(r_1,r_2)}(\sigma)$ (see (\ref{eqn:3.1})). 
Furthermore, 
under the assumption that $q$ is even, 
$f(x,y)=f(|x|,|y|)$ for any $(x,y)\in\R^2$, 
which implies  
\begin{equation}\label{eqn:5.1}
Z_{U(r_1,r_2)}(\sigma)=4Z_{U_+(r_1,r_2)}(\sigma)=4Z(\sigma).
\end{equation} 
Therefore, Theorem~2.1 can be proved by using 
Theorem~3.1 as follows.

\vspace{.5 em}

{\bf The cases (i), (ii).}\quad
For any $\epsilon>0$, 
there exist $r_1, r_2\in(0,1)$ such that
\begin{equation*}
\varphi(0,0)-\epsilon \leq 
\varphi(x,y)\leq 
\varphi(0,0)+\epsilon
\quad \mbox{ for $(x,y)\in U(r_1,r_2)$.}
\end{equation*}
These inequalities imply that
\begin{equation}\label{eqn:5.2}
\begin{split}
(\varphi(0,0)-\epsilon) \cdot 
Z_{{U(r_1,r_2)}}(\sigma)
&\leq 
Z_f(\varphi)(\sigma)-
\int_{U\setminus U(r_1,r_2)}|f(x,y)|^{\sigma}\varphi(x,y)dxdy\\
&\leq
(\varphi(0,0)+\epsilon) \cdot 
Z_{U(r_1,r_2)}(\sigma).
\end{split}
\end{equation}
We remark that the integral in (\ref{eqn:5.2}) converges
and is bounded by a positive constant
which is independent of $\sigma$
since $f$ does not vanish on $U\setminus U(r_1,r_2)$.

As a result, 
the inequalities (\ref{eqn:5.2}) and 
Theorem~3.1 easily imply (i), (ii) in Theorem~2.1.

\vspace{.5 em}
{\bf The case (iii).} \quad 
We define
\begin{equation*}
\varphi_+(x,y)
=\max\{\varphi(x,y),0\}
\,\,\, \mbox{ and  }\,\,\,
\varphi_-(x,y)
=\max\{-\varphi(x,y),0\}.
\end{equation*}
Of course, 
$\varphi(x,y)=\varphi_+(x,y)-\varphi_-(x,y)$ holds.

Now, 
let $R_1,R_2$ be positive constants such that
the support of $\varphi$ is contained in 
$U(R_1,R_2)$. 
Then, Theorem~3.1 (iii) implies that 
there exist  $\delta>0$ and $C_{\pm}(R_1,R_2)>0$,
which depends on $R_1, R_2$ and is independent of $\sigma$, 
such that 
\begin{equation}\label{eqn:5.3} 
\begin{split}
Z_f(\varphi_{\pm})(\sigma)
&\leq
\max_{(x,y)\in U} \varphi_{\pm}(x,y) \cdot 
Z_{{\rm Supp}(\varphi_{\pm})}(\sigma) \\
&\leq
\max_{(x,y)\in U} \varphi_{\pm}(x,y) 
\cdot Z_{U(R_1,R_2)}(\sigma) \\
&\leq 
C_{\pm}(R_1,R_2)\cdot
\max_{(x,y)\in U} \varphi_{\pm}(x,y),
\end{split}
\end{equation}
for $\sigma\in(-1/b,-1/b+\delta)$.
Since
$Z_f(\varphi_{\pm})(\sigma)$ are monotone decreasing functions 
of $\sigma\in(-1/b,0)$,
the above estimates easily imply that 
there exist nonnegative constants 
$B_+(\varphi)$ and $B_-(\varphi)$ such that 
\begin{equation*}
\lim_{\sigma\to-1/b+0} Z_f(\varphi_{\pm})(\sigma)=B_{\pm}(\varphi).
\end{equation*}
Let $B(\varphi):=B_+(\varphi)-B_-(\varphi)$, then
we can get the limit in (iii).
Note that when $\varphi$ satisfies (\ref{eqn:1.4}),
$B(\varphi)=B(\varphi_+)$ and $B(\varphi_-)=0$ hold, 
so $B(\varphi)$ is positive.

\begin{remark}
The inequalities in (\ref{eqn:5.3}) imply the continuous 
property for the distribution defined by the map:
$\varphi\mapsto 
B(\varphi)=\lim_{\sigma\to-1/b} Z_f(\varphi)(\sigma)$.
\end{remark}

\section{Landau type theorem for local zeta functions}

In this section, we deal with $Z_f(\varphi)(s)$
in the general dimensional case, 
i.e. 
\begin{equation}\label{eqn:6.1}
Z_f(\varphi)(s)=\int_{\mathbb{R}^n}|f(x)|^s \varphi(x)dx
\quad\quad \mbox{for $s\in \C$,}
\end{equation}
where
$f,\varphi$ are real-valued smooth functions 
defined on an open neighborhood $U$ of the origin 
in $\mathbb{R}^n$
and the support of $\varphi$ is contained in $U$.

Lemma~1.1  in the Introduction follows from 
the following theorem.

\begin{theorem}\label{prop:3.1}
Suppose 
that $f(0)=0$, $|f(x)|<1$ and $\varphi(x)\geq 0$ on $U$.
Let $\rho$ be nonpositive number such that
the integral $Z_f(\varphi)(s)$ in (\ref{eqn:6.1}) 
converges if ${\rm Re}(s)>\rho$. 
If 
$Z_f(\varphi)$ can be analytically continued 
as a holomorphic function
to some open neighborhood of $s=\rho$,
then there exists a positive number $\delta$
such that the integral $Z_f(\varphi)(\rho-\delta)$
converges.
\end{theorem}

Indeed, 
if $Z_f(\varphi)(s)$ can be holomorphically continued across
the point $s=-c_0(f)$,
then 
the above theorem implies that
$Z_f(\varphi)(s)$ becomes a holomorphic function 
on the region: ${\rm Re}(s)>-c_0(f)-\delta$ 
with some positive $\delta$, which
is a contradiction to (\ref{eqn:1.5}). 

The above property of local zeta functions
itself is interesting and 
is analogous to Landau's theorem on the Dirichlet series
$\sum_{n=1}^{\infty} a_n n^{-s}$
where $a_n$ are nonnegative numbers
(see \cite{Zag81}).

\begin{proof}
From the assumption,
$Z_f(\varphi)$ can be considered as a holomorphic function
on the region: ${\rm Re}(s)>\rho$ and, moreover, there exist 
an open neighborhood $V$ of $s=\rho$ and 
a holomorphic function $\tilde{Z}$ defined 
on the set $V\cup\{s\in \mathbb{C}: {\rm Re}(s)>\rho\}$ 
such that 
$\tilde{Z}=Z_f(\varphi)$ on the region: ${\rm Re}(s)>\rho$.

Now, there exists a positive number $\delta$
such that
the disc: 
$$
D:=\{z\in\C:|z-(\rho+1)|<1+2\delta\}
$$
is contained in the region 
$V\cup\{s\in \mathbb{C}: {\rm Re}(s)>\rho\}$.
Since $\tilde{Z}$ is holomorphic on $D$,
its Taylor series converges to the value of $\tilde{Z}(s)$ 
at any point of $D$,
i.e.,
\begin{equation}\label{eqn:6.2}
\tilde{Z}(s)=\sum_{j=0}^{\infty}
\frac{1}{j!}\frac{d^j\tilde{Z}}{ds^j}(\rho+1)(s-(\rho+1))^j
\quad \mbox{ for $s\in D$.}
\end{equation}
Since the point $s=\rho-\delta$ is contained 
in $D$, 
we have
\begin{equation}\label{eqn:6.3}
\tilde{Z}(\rho-\delta)=\sum_{j=0}^{\infty}
\frac{1}{j!}\frac{d^j \tilde{Z}}{ds^j}(\rho+1)(-\delta-1)^j
\end{equation}
and this series converges.
On the other hand, we have
\begin{eqnarray}\label{eqn:6.4}
\frac{d^j \tilde{Z}}{d s^j}(s)=
\int_{\mathbb{R}^n}|f(x)|^s (\log |f(x)|)^{j}
\varphi(x)dx \quad \mbox{for $j\in\N$}
\end{eqnarray}
if $s$ satisfies ${\rm Re}(s)>\rho$.
Indeed, it is easy to show the possibility of
the exchange of integral and derivatives. 
Substituting (\ref{eqn:6.4}) to (\ref{eqn:6.3}), 
we have
\begin{equation}\label{eqn:6.5}
\begin{split}
\tilde{Z}(\rho-\delta)
&=\sum_{j=0}^{\infty}
\frac{1}{j!}(-\delta-1)^j
\int_{\mathbb{R}^n}|f(x)|^{\rho+1} (\log |f(x)|)^{j}\varphi(x)dx\\
&=
\sum_{j=0}^{\infty}
\int_{\mathbb{R}^n}|f(x)|^{\rho+1}
\frac{1}{j!}
((-\delta-1)\log |f(x)|)^{j}\varphi(x)dx.
\end{split}
\end{equation}
Since 
all terms of the series in \eqref{eqn:6.5} are positive
and
the series
converges, 
the order of the summation and the integral 
can be exchanged. 
Therefore, 
\begin{equation*}
\begin{split}
\tilde{Z}(\rho-\delta)
&=
\int_{\mathbb{R}^n}|f(x)|^{\rho+1}
\left(
\sum_{j=0}^{\infty}\frac{1}{j!}
((-\delta-1)\log |f(x)|)^{j}
\right)
\varphi(x)dx \\
&=
\int_{\mathbb{R}^n}|f(x)|^{\rho+1}e^{(-\delta-1)\log|f(x)|}\varphi(x)dx\\
&=
\int_{\mathbb{R}^n}|f(x)|^{\rho-\delta}\varphi(x)dx.
\end{split}
\end{equation*}
The above equalities imply that 
the last integral converges. 
\end{proof}

\section{Discussion and open questions}

\subsection{Singularities of $Z_f(\varphi)(s)$}

First, 
let us consider the case when
$f$ is as in (\ref{eqn:2.1}) and $\varphi$ satisfies 
the condition (\ref{eqn:1.4}). 
As mentioned in Collorary 2.2, 
the singularity of $Z_f(\varphi)(s)$ at $s=-1/b$ 
is different from a pole. 
To be more elementary, 
the following question is naturally raised. 

\begin{question}
Is the singularity of $Z_f(\varphi)$ at $s=-1/b$ isolated or not?
\end{question}

If this singularity was isolated, 
then it must be an essential singularity. 
At present, 
it seems impossible to answer this question 
from the information from Theorem~2.1 only. 

Next, let us consider  
more global property of $Z_f(\varphi)(s)$. 
It should be expected that
$Z_f(\varphi)(s)$ can be holomorphically extended 
to a wider domain containing the region: ${\rm Re}(s)>-1/b$.
The following question is considered 
as a first step to this problem. 

\begin{question}
Does $Z_f(\varphi)$ have another singularity on 
the vertical line: ${\rm Re}(s)=-c_0(f)$?
\end{question}


\subsection{Openness problem}

Let $f$ be a smooth function with the condition 
(\ref{eqn:1.2}). 
Let us consider the following subset in $\R$:
\begin{equation*}
H(f):=
\left\{\mu>0:
\begin{array}{l} 
\mbox{there exists an open neighborhood $V$ of} \\
\mbox{the origin in $U$ such that $|f|^{-\mu}\in L^1(V)$}
\end{array}
\right\}.
\end{equation*}
(Of course, $c_0(f)=\sup H(f)$.)
When $f$ is real analytic, the set $H(f)$ is {\it open} in $\R$
from the fact that $s=-c_0(f)$ is a pole for $Z_f(\varphi)$
for $\varphi$ satisfying (\ref{eqn:1.4}).
Without the real analyticity assumption, 
our observation implies that 
$H(f)$ is not always open.  
More precisely, in the case of (\ref{eqn:2.1}),  
the openness of $H(f)$ depends on 
the relationship among the parameters $a,b,p$. 
Generally, when $f$ is a smooth function, 
the following question seems interesting. 

\begin{question}
Which condition on $f$ gives the openness (or closedness)
of $H(f)$? 
\end{question}

An analogous problem has been deeply investigated 
in the case of complex variables from the viewpoint of 
complex geometry. 
The {\it openness conjecture}, 
raised by Demailly and Koll\'ar \cite{DeK01}, 
is the following: 
``If $\phi$ is plurisubharmonic, 
then the set $H(e^{-\phi})$ is always open.''   
This conjecture has been affirmatively solved 
in \cite{FaJ05jams},  \cite{FaJ05inv},  \cite{Ber13}. 
Our observation indicates that the openness of $H(f)$ needs 
some kind of good property of $f$.

\subsection{Oscillatory integrals}
Let us consider an oscillatory integral of the form:
\begin{equation*}
I_f(\varphi)(\tau):=
\int_{\R^2}e^{i \tau f(x,y)}\varphi(x,y)dxdy 
\quad \mbox{for $\tau>0$},
\end{equation*}
where  $f$, $\varphi$ are as in (\ref{eqn:1.1}) 
and they satisfy the conditions 
(\ref{eqn:1.2}), (\ref{eqn:1.4}).

It is known (see \cite{Var76}, \cite{AGV88}, etc.) 
that there is a deep relationship
between the behavior of oscillatory integrals at infinity
and the distribution of poles of local zeta functions. 
Indeed, the Mellin transformation gives a
clear relationship between oscillatory integrals 
and some functions similar to local zeta functions. 

First, let us consider the case when $f$ is real analytic. 
As mentioned in the Introduction, 
the integral $Z_f(\varphi)(s)$ can be
analytically continued as
a meromorphic function to the whole complex plane. 
Furthermore, under some assumption,   
its leading pole exists at $s=-1/d(f)$ and 
its order is $m(f)$.  
(Note that $m(f)$ is the positive integer 
determined by some topological information 
of the Newton polyhedron of $f$.) 
By using this fact, we have
\begin{equation}\label{eqn:7.1}
\lim_{\tau\to+\infty}
\tau^{1/d(f)}(\log \tau)^{-m(f)+1}
I_f(\varphi)(\tau)=C_f(\varphi),
\end{equation}
where $C_f(\varphi)$ is a positive constant independent of 
$\tau$. 

Next, let us consider the smooth case. 
Although the formula (\ref{eqn:7.1}) 
can be directly generalized in many smooth cases 
(\cite{Gre09}, \cite{IkM11jfaa}, \cite{KaN16jmsut}), 
there exist examples for which
the behavior of $I_f(\varphi)$ is different from (\ref{eqn:7.1}). 
In \cite{KaN17},
the authors investigate the case when
\begin{equation}\label{eqn:7.2}
f(x,y)=y^b+e^{-1/|x|^p},
\end{equation}
where $p>0$ and
$b\in\Z$ with $b\geq 2$, and obtain the following:
\begin{equation*}
\lim_{\tau\to+\infty}\tau^{1/b}(\log \tau)^{1/p}\cdot
I_f(\varphi)(\tau)=C_b\varphi(0,0),
\end{equation*}
where $C_b$ is a nonzero constant depending only on
$b$ and is explicitly computed. 
We remark that $d(f)=b$ and $m(f)=1$ in this case. 
Since (\ref{eqn:7.2}) is a special case of (\ref{eqn:2.1}) 
($a=0$, $b=q$), 
Theorem~2.1 implies that if $p>1$, then
\begin{equation*}
\lim_{\sigma\to -1/b+0}
(b\sigma+1)^{-1/p+1}\cdot Z_f(\varphi)(\sigma)=
C\varphi(0,0),
\end{equation*} 
where $C$ is a positive constant. 

Now, we are interested in how flat functions affect
the behavior of $I_f(\varphi)(\tau)$.
Observing the case when $f$ is real analytic or
$f$ is as in (\ref{eqn:7.2}), 
one can easily recognize some correspondence between the bahaviors of
$Z_f(\varphi)(\sigma)$ and $I_f(\varphi)(\tau)$.  
It will be valuable to affirmatively answer the following question. 
\begin{question}
When a positive limit of 
$(\sigma+c_0(f))^{\alpha}Z_f(\varphi)(\sigma)$
as $\sigma\to-c_0(f)+0$ exists, 
does a nonvanishing limit of 
$\tau^{c_0(f)}(\log \tau)^{-\alpha+1}\cdot I_f(\varphi)(\tau)$ 
as $\tau\to+\infty$ exist?

In particular,
let $f$ be as in (\ref{eqn:2.1}) with 
$p>1-a/b$, then 
does the following hold?
\begin{equation*}
\lim_{\tau\to +\infty}
\tau^{1/b}(\log \tau)^{\frac{1-a/b}{p}}\cdot I_f(\varphi)(\tau)=
C\varphi(0,0),
\end{equation*}
where $C$ is a positive constant
which is independent of $\tau$.
\end{question}

\vspace{1 em}

{\sc Acknowledgements.}\quad 
The authors would like to express their sincere gratitude 
to the referee for giving the authors valuable comments. 
This work was partially supported by 
JSPS KAKENHI Grant Numbers 
JP15K04932, JP15K17565, JP15H02057.



\begin{thebibliography}{99}
\bibitem{AGV88}
V. I. Arnold, S. M. Gusein-Zade and A. N. Varchenko: 
   {\it Singularities of Differentiable Maps II}, 
   Birkh$\ddot{{\rm a}}$user, 1988. 
%
\bibitem{Ati70}
M. F. Atiyah:
Resolution of singularities and division of distributions, 
Comm. Pure Appl. Math. {\bf 23} (1970), 145--150. 

\bibitem{Ber13}
B. Berndtsson: 
The openness conjecture for plurisubharmonic functions, 
2013. arXiv 1305.5781.


\bibitem{BeG69}
I. N. Bernstein and S. I. Gel'fand:
Meromorphy of the function $P^{\lambda}$,  
Funktsional. Anal.  Prilozhen. {\bf 3} (1969), 
84--85. 








\bibitem{CGP13}
T. C. Collins, A. Greenleaf and M. Pramanik:
A multi-dimensional resolution of 
singularities with applications to analysis, 
Amer. J. Math.  {\bf 135}  (2013), 
1179--1252. 

\bibitem{DeK01}
J.-P. Demailly and J. Koll\'ar: 
Semi-continuity of complex singularity exponents
and K\"ahler-Einstein metrics on Fano orbifolds, 
Ann. Sci. \'Ecole Norm. Sup. 
{\bf 34} (2001), 525--556.


\bibitem{DNS05} 
J. Denef, J. Nicaise and P. Sargos: 
Oscillating integrals and Newton polyhedra, 
J. Anal. Math. {\bf 95} (2005), 147--172. 

\bibitem{DS89}
J. Denef and P. Sargos:
Poly\`edra de Newton et distribution $f_+^s$. I,
J. Anal. Math. {\bf 53} (1989), 201--218.

\bibitem{DS92}
\bysame:
Poly\`edra de Newton et distribution $f_+^s$. II, 
Math. Ann. {\bf 293} (1992), 193--211. 

\bibitem{FaJ05jams}
C. Favre and M. Jonsson:  
Valuations and multiplier ideals, 
J. Amer. Math. Soc. {\bf 18} (2005), 655--684. 

\bibitem{FaJ05inv}
\bysame: 
Valuative analysis of planar plurisubharmonic functions,
Invent. Math. {\bf 162} (2005), 271--311.



\bibitem{GeS64}
I. M. Gel'fand and G. E. Shilov: 
{\it Generalized functions --- properties and operations}, 
Volume I, Academic Press, 1964.

\bibitem{Gre06}
M. Greenblatt:
Newton polygons and local integrability of 
negative powers of smooth functions in the plane, 
Trans. Amer. Math. Soc.  {\bf 358}  (2006),  657--670. 

\bibitem{Gre09}
\bysame:
The asymptotic behavior of degenerate oscillatory integrals in two dimensions,
J. Funct. Anal. {\bf 257} (2009), 1759--1798.

\bibitem{Hir64}
H. Hironaka: 
   Resolution of singularities of an algebraic variety 
   over a field of characteristic zero I, II, Ann. of Math. 
   {\bf 79} (1964), 109--326. 



%

\bibitem{IkM11tams}
I. A. Ikromov and D. M\"uller: 
On adapted coordinate systems, 
Trans. Amer. Math. Soc. {\bf 363} (2011), 2821--2848. 

\bibitem{IkM11jfaa}
\bysame:
Uniform estimates for the Fourier transform of surface carried measures 
in $\mathbb{R}^3$ and an application to Fourier restriction, 
J. Fourier Anal. Appl. {\bf 17} (2011), 1292--1332.





\bibitem{KaN16jmsut}
J. Kamimoto and T. Nose:
Toric resolution of singularities in a certain class 
of $C^{\infty}$ functions and asymptotic analysis of 
oscillatory integrals, 
J. Math. Soc. Univ. Tokyo, {\bf 23} (2016), 425--485.

\bibitem{KaN16tams}
\bysame:
Newton polyhedra and weighted oscillatory integrals
with smooth phases, 
Trans. Amer. Math. Soc. {\bf 368} (2016), 5301--5361.

\bibitem{KaN17}
\bysame:
Asymptotic limit of oscillatory integrals with certain smooth phases, 
RIMS K\^oky\^uroku Bessatsu {\bf B63} (2017), 103--114.

\bibitem{OkT13}
T. Okada and K. Takeuchi:
Coefficients of the poles of local zeta functions and 
their applications to oscillating integrals,
Tohoku Math. J.   {\bf 65}  (2013),  159--178. 


\bibitem{PSS99}
D. H. Phong, E. M. Stein and J. A. Sturm: 
On the growth and stability of real-analytic functions, 
Amer. J. Math. {\bf 121} (1999), 519--554.


\bibitem{Var76}
A. N. Varchenko: 
   Newton polyhedra and estimation of oscillating 
   integrals, Functional Anal. Appl., {\bf 10-3} (1976), 175--196.

\bibitem{Zag81}
D. B. Zagier: 
{\it Zetafunktionen und quadratische K\"orper}, (German),
Springer-Verlag, Berlin-New York, 1981. 


\end{thebibliography}
\end{document}